\documentclass[reqno]{amsart}

\usepackage{amsmath,amssymb,amsfonts,amsthm}

\theoremstyle{plain}
\newtheorem{theorem}{Theorem}
\newtheorem{lemma}{Lemma}

\theoremstyle{definition}
\newtheorem{definition}{Definition}

\title{A class of non-cylindrical domains for parabolic equations}

\author{Alberto Dom\'{\i}nguez Corella}
\address{Departamento de Matem\'aticas, Centro de Investigaci\'on y Estudios Avanzados, Ciudad de M\'exico, M\'exico}
\email{adominguez@math.cinvestav.mx}

\author{Jorge Rivera Noriega}
\address{Centro de Investigaci\'on en Ciencias, Universidad Aut\'onoma del Estado de Morelos, Cuernavaca, Morelos, M\'exico}
\email{rnoriega@uaem.mx}

\subjclass[2010]{Primary 26B10, 26B35; Secondary 53A05}
\keywords{Implicit function theorem, Lip(1,1/2) functions, non-cylindrical domains}

\begin{document}
\begin{abstract}
	We present a class of non-cylindrical domains where Dirichlet-type problems for parabolic equations, such as the heat equation, can be posed and solved. The
regularity for the boundary of this class of domains is a mixed Lipschitz condition, as
described in the bulk of the paper. The main tool is an adequate version of the implicit
function theorem for functions with this kind of regularity. It is proved that the class
introduced herein is of the same type as domains previously considered by several
authors.
\end{abstract}
\maketitle

\vspace{-1.2em}
\begin{center}
	\small
	\textbf{Note.} This pre-print is the English translated version of the article:
	Dom\'inguez Corella, Alberto and Rivera Noriega, Jorge,
	\emph{A class of non-cylindrical domains for parabolic equations},
	\textit{Lecturas Matem\'aticas} \textbf{38} (2017), no.~2, 49--63,
	which was originally written and published in Spanish.
	The mathematical content is identical to the original.
\end{center}

\section{Motivation}

In several relatively recent works, Dirichlet-type problems associated with the heat
equation, or more generally with second order parabolic equations, have been considered
on non-cylindrical domains whose boundary satisfies a local Lipschitz $(1,1/2)$
regularity condition, to be described precisely below. An incomplete list of references,
ordered roughly chronologically, includes
\cite{Kemper1972,Brown1989,LewisMurray1992,LewisMurray1995,HofmannLewis1996,
HofmannLewis2005,HofmannLewis1999,Nystrom1997,BrownHuLieberman1997,
RiveraNoriega2003,HofmannLewisNystrom2004,LewisNystrom2007,Nystrom2008,
ArgiolasGrimaldi2010,RiveraNoriega2014,ChoDongKim2015,DindosPetermichlPipher}
and the monograph \cite{Lieberman1996}.

Our aim is to define a class of non-cylindrical domains in $\mathbb{R}^{n+1}$ on which
the main results in these references remain valid, and whose boundary admits a fairly
direct description in a precise sense.

\subsection*{Two classes of Lipschitz domains}

\begin{definition}
    An open bounded set $\Omega\subset\mathbb{R}^{n}$ is a star-like Lipschitz
domain (centered at the origin) if, in spherical coordinates on $\mathbb{R}^{n}$, it can
be written as
\[
\Omega = \bigl\{ r\,\omega : \omega\in S^{n-1},\ 0\le r < \varphi(\omega)\bigr\},
\]
where $S^{n-1}$ denotes the unit sphere in $\mathbb{R}^{n}$ and
$\varphi:S^{n-1}\to(0,\infty)$ is a Lipschitz function, that is, there exists a constant
$M>0$ such that
\[
|\varphi(x)-\varphi(y)| \le M\,|x-y|\quad\text{for }x,y\in S^{n-1}.
\]
\end{definition}
In this definition, and subsequently, $|x|$ denotes the Euclidean norm of
$x\in\mathbb{R}^{n}$; the context prevents confusion with the absolute value on
$\mathbb{R}$.

As an example, one may consider the domain in $\mathbb{R}^{2}$ whose boundary is
described, in polar coordinates, by
\[
\bigl\{ r\,\omega \in \mathbb{R}^{2} : r = \varphi(\omega) = 2+\cos\omega,\ 0\le \omega<2\pi \bigr\}.
\]
Observe that, since $\bar\Omega = \bigl\{ r\,\omega : \omega\in S^{n-1},\ 0\le r \le \varphi(\omega)\bigr\}$
and $\varphi$ is continuous, we have
\[
\partial\bar\Omega
= \bigl\{\varphi(\omega)\,\omega: \omega\in S^{n-1}\bigr\}
= \partial\Omega.
\]

\begin{definition}
    We say that an open set $\Omega\subset\mathbb{R}^{n}$ with
$\partial\Omega = \partial\bar\Omega$ is a Lipschitz domain if for each
$x_{0}\in\partial\Omega$ there is a new coordinate system obtained from the canonical one
by a composition of translations and rotations, together with a cylinder
$C = B\times I$ (where $B$ is a ball in $\mathbb{R}^{n-1}$ and $I$ is an interval) that
contains $x_{0}$, and a Lipschitz function $\psi:\mathbb{R}^{n-1}\to\mathbb{R}$ whose
graph $\Sigma(\psi,\mathbb{R}^{n-1})
=\bigl\{(x,\psi(x)) : x\in\mathbb{R}^{n-1}\bigr\}$ 
contains $x_{0}$, satisfies $\psi(B)\subset I$, and such that $C\cap\partial\Omega = C\cap\Sigma(\psi,\mathbb{R}^{n-1}).$
\end{definition}

This definition is taken from \cite{Wuertz2008}, but it already appears in earlier
references such as \cite{HuntWheeden1968,Dahlberg1979,Verchota1984}. See also the
survey \cite{Kenig1986} and the more recent work \cite{HofmannMitreaTaylor2007}.

The next result gives the connection between these two kinds of domains.

 \begin{theorem}
     If $\Omega$ is a star-like Lipschitz domain, then it is a Lipschitz domain.
 \end{theorem}
The proof of this theorem is given in \cite{Wuertz2008}, and is obtained by means of
the following version of the implicit function theorem, also proved there.

\begin{theorem}[Implicit function theorem for Lipschitz functions]
    Let $U_{m}\subset\mathbb{R}^{m}$ and $U_{n}\subset\mathbb{R}^{n}$ be open sets. For
$a\in U_{m}$ and $b\in U_{n}$ let
$f:U_{m}\times U_{n}\to\mathbb{R}^{n}$ be a function such that $f(a,b)=0$, and assume
that there exists $K_{1}>0$ such that
\begin{equation}\label{eq:Lip-f}
|f(x_{1},y_{1}) - f(x_{2},y_{2})|
\le K_{1}\,\bigl|(x_{1}-x_{2},\,y_{1}-y_{2})\bigr|
\end{equation}
for all $(x_{1},y_{1}),(x_{2},y_{2})\in U_{m}\times U_{n}$. Suppose in addition that
\begin{equation}\label{eq:nondeg-f}
|f(x,y_{1}) - f(x,y_{2})|
\ge K_{2}\,|y_{1}-y_{2}|
\end{equation}
for some $K_{2}>0$ and all $(x,y_{1}),(x,y_{2})\in U_{m}\times U_{n}$. Then there exists
an open set $V_{m}\subset\mathbb{R}^{m}$ such that $a\in V_{m}\subset U_{m}$ and a
unique function $\varphi:V_{m}\to U_{n}$ satisfying a Lipschitz condition of the form
\begin{equation}\label{eq:Lip-phi}
|\varphi(x_{1}) - \varphi(x_{2})|
\le K\,|x_{1}-x_{2}|\quad\text{for }x_{1},x_{2}\in V_{m},
\end{equation}
for some $K>0$ and
\[
\{(x,y)\in V_{m}\times U_{n} : f(x,y)=0\}
=
\{(x,\varphi(x)) : x\in V_{m}\}.
\]
\end{theorem}\vspace{4pt}

\noindent
\textbf{General strategy of the present work}
\vspace{2pt}

By Theorem~1, the crucial local estimates leading to the solution of $L^{p}$ Dirichlet
problems for the Laplacian, which are known to hold on Lipschitz domains, also hold on
star-like Lipschitz domains. A general description of the harmonic measure method for
solving $L^{p}$ Dirichlet problems can be found in \cite[pp.~141--143]{Kenig1986}, and a
more detailed treatment appears, for instance, in
\cite{Dahlberg1977,Dahlberg1979,FeffermanKenigPipher1991}.

Recall that, as mentioned at the beginning of this section, part of our purpose is to
ensure that some results for parabolic equations, valid in a certain class of non-cylindrical
domains, also hold for the domains that we introduce in this paper.

We shall therefore try to describe, by means of cylindrical coordinates in
$\mathbb{R}^{n+1}$, a suitable analogue of star-like Lipschitz domains, and we shall prove
an analogue of Theorem~1 (see Theorem~6), which will show that the results in the
references mentioned above remain valid for our class of domains.

The details of these adaptations are carried out in the next section, and are based on
ideas from \cite{Wuertz2008}. In particular, the analogue of Theorem~1 will again be
proved using an appropriate version of the implicit function theorem (see Theorem~5).
\vspace{2pt}

\noindent
\textbf{Remarks on Theorem~2}\vspace{2pt}

It is worth emphasizing that this version of the implicit function theorem does not
require the functions involved to be of class $C^{1}$, as in the classical versions usually
presented in standard analysis textbooks (see, for example, \cite{Apostol1981}). Its
statement, however, retains certain essential features.

\begin{itemize}
    \item \textit{Regularity.} In this setting, both the function $f$ and the implicit function $\varphi$
satisfy Lipschitz-type conditions such as \eqref{eq:Lip-f} and \eqref{eq:Lip-phi}.

    \item \textit{Non-degeneracy.} Instead of assuming that a certain Jacobian determinant is non-zero,
one requires the property \eqref{eq:nondeg-f}.

    \item \textit{Uniqueness.} The implicit function $\varphi$ obtained is unique.
\end{itemize}
Moreover, this version of the implicit function theorem, in which regularity hypotheses
are relaxed, is by no means optimal. We refer the interested reader to the monograph
\cite{KrantzParks2013}, and to the more recent research article \cite{AzzamSchul2012} for
more refined results.

Even so, in the next section we shall have the opportunity to present the main ideas of
these proofs, when we adapt them to establish our main theorems.

\section{Basic definitions and main results}

Whenever the Euclidean space $\mathbb{R}^{n+1}$ or $\mathbb{R}^{n}$ contains the
variable $t$ (which, in problems related to the heat equation, plays the role of the time
variable), we shall equip it with a new metric that is compatible with a certain \textit{non–isotropic
dilation}, to be described below. With this new \textit{homogeneity} we shall also introduce an
adapted notion of Lipschitz–type functions together with the corresponding implicit function
theorem, and then define the new class of non–cylindrical domains.

\subsection{Parabolic homogeneity in $\mathbb{R}^{n+1}$}

We use Doob's notation (see \cite[Chapter~XV]{Doob2001}) for Euclidean
spaces that contain the time-variable:
\[
\dot{\mathbb{R}}^{n+1}=\{(t,x): t\in\mathbb{R},\ x\in\mathbb{R}^{n}\},\qquad
\dot{\mathbb{R}}^{n}=\{(t,x'): t\in\mathbb{R},\ x'\in\mathbb{R}^{n-1}\}.
\]
To simplify some subsequent statements, we shall occasionally use similar notation for
points or subsets of these spaces; for instance, we may explicitly write
$\dot U\subset\dot{\mathbb{R}}^{n+1}$ to emphasize that $U$ is a subset of an $(n+1)$–dimensional
Euclidean space that includes the time variable.

To motivate the use of a change of homogeneity in $\dot{\mathbb{R}}^{n+1}$, let
$u(t,x)$ be a solution of the heat equation $Hu=0$, where the heat operator is given by
\[
Hu(t,x)=\frac{\partial u}{\partial t}
-\sum_{j=1}^{n}\frac{\partial^{2}u}{\partial x_{j}^{2}},
\qquad x=(x_{1},\dots,x_{n}).
\]
Given $\lambda>0$ we consider
\[
v(t,x)=u(\lambda^{\alpha}t,\lambda^{\beta}x),
\]
and we ask for which values $\alpha,\beta>0$ the function $v$ is also a solution of
$Hv=0$. A straightforward computation shows that we must have $2\beta=\alpha$.

Observe that the value $\beta=1$ corresponds to the usual (isotropic) dilation in
$\mathbb{R}^{n}$, and that this choice naturally yields $\alpha=2$. With these values we
obtain a (non–isotropic) dilation in $\dot{\mathbb{R}}^{n+1}$ that induces a change of
homogeneity: given $\lambda>0$ we define
\begin{equation}\label{eq:parabolic-dilation}
T_{\lambda}(t,x)=(\lambda^{2}t,\lambda x)
=
\begin{pmatrix}
\lambda^{2} & \vec{0} \\
\vec{0} & \lambda \mathbb I_{n}
\end{pmatrix}
\begin{pmatrix}
t \\[1mm] x
\end{pmatrix},
\end{equation}
where $\mathbb I_{n}$ denotes the $n\times n$ identity matrix and $\vec{0}$ denotes the zero
row or column vector in $\mathbb{R}^{n}$. Note that the operator norm of the matrix in
\eqref{eq:parabolic-dilation}, regarded as a linear map from $\mathbb{R}^{n+1}$ into itself,
is
\[
\|T_{\lambda}\|=\max\{\lambda,\lambda^{2}\}.
\]

A way to equip $\dot{\mathbb{R}}^{n+1}$ with a metric adapted to this and more general
changes of homogeneity was introduced by B.~F.~Jones \cite{Jones1964} and
E.~Fabes and N.~Rivi\`ere \cite{FabesRiviere1966,Riviere1971} in their study of singular
integrals associated with the heat equation. We briefly recall some of those ideas, which
in the original references are developed for a more general family of dilations than the
$T_{\lambda}$ defined above.

First, note that one way to obtain the usual Euclidean norm of $x\in\mathbb{R}^{n}$,
$x\neq \vec 0$, is through the solution $r$ of the equation
\[
\sum_{j=1}^{n}\frac{x_{j}^{2}}{r^{2}}=1,
\qquad x=(x_{1},\dots,x_{n}).
\]
To introduce the change of homogeneity in $\dot{\mathbb{R}}^{n+1}$, given
$(t,x)\in\dot{\mathbb{R}}^{n+1}\setminus\{(0,0)\}$ we look for a solution $\rho$ of the equation
\[
\frac{t^{2}}{\rho^{4}}+\sum_{j=1}^{n}\frac{x_{j}^{2}}{\rho^{2}}=1,
\qquad x=(x_{1},\dots,x_{n}).
\]
To justify the existence of such solution, we consider the function $F:\bigl(\dot{\mathbb{R}}^{n+1}\setminus\{(0,\vec0)\}\bigr)\times(0,\infty)\to\mathbb{R}$
defined by
\[
F(t,x,\rho)=\frac{t^{2}}{\rho^{4}}+\sum_{j=1}^{n}\frac{x_{j}^{2}}{\rho^{2}}.
\]
Then $F$ is continuous and strictly decreasing in the variable $\rho$, and in addition
\begin{equation}\label{eq:F-limits}
\lim_{\rho\to\infty}F(t,x,\rho)=0,
\qquad
\lim_{\rho\to 0^{+}}F(t,x,\rho)=\infty.
\end{equation}
Consequently, for every $(t,x)\in\dot{\mathbb{R}}^{n+1}\setminus\{(0,\vec0)\}$ there exists a
unique $\rho=\rho(t,x)$ such that $F(t,x,\rho)=1$. Defining also $\rho(0,\vec0)=0$ we are
now in a position to define a metric on $\dot{\mathbb{R}}^{n+1}$.

We first observe that a direct consequence of this definition is that
\[
\rho\bigl(T_{\lambda}(t,x)\bigr)=\lambda\,\rho(t,x)
\quad\text{for all }(t,x)\in\mathbb{R}^{n+1},\ \lambda>0.
\]
Indeed, note that for all $x\in\mathbb{R}^{n}$, $t\in\mathbb{R}$, $\rho>0$ and $\lambda>0$
we have $F(\lambda^{2}t,\lambda x,\lambda\rho)=F(t,x,\rho)$. 
Fix $t\in\mathbb{R}$ and $x\in\mathbb{R}^{n}$, and let $\rho_{1}$ be the unique solution
of $F(t,x,\rho_{1})=1$. By the previous  identity we then have
$F(\lambda^{2}t,\lambda x,\lambda\rho_{1})=1$, which shows that $\rho_{2}=\lambda\rho_{1}$
is the solution of $F(T_{\lambda}(t,x),\rho_{2})=1$, that is,
$\rho(T_{\lambda}(t,x))=\lambda\rho(t,x)$, as claimed.

Another immediate consequence of the definition is that
\[
\bigl|T_{1/\lambda}(t,x)\bigr|^{2}=F(t,x,\lambda),
\qquad x\in\mathbb{R}^{n},\ t\in\mathbb{R},\ \lambda>0.
\]
Recall that $|\cdot|$ denotes the Euclidean norm in $\mathbb{R}^{n}$. In the previous  identity, and in what
follows, we also use $|\cdot|$ for the Euclidean norm in $\mathbb{R}^{n+1}$.

The next result confirms the intuition that $\rho$ behaves in many ways like a norm;
it can be found, in a more general setting, in \cite[Theorem~7.1]{Riviere1971}.

\begin{theorem}\label{thm:parabolic-metric}
The function
\[
D\bigl((t,x),(s,y)\bigr)=\rho(t-s,x-y)
\]
defines a metric on $\dot{\mathbb{R}}^{n+1}$.
\end{theorem}

Before giving the proof, we point out that the same idea can be applied in the space
$\dot{\mathbb{R}}^{n}$ (with the time variable  still present as a distinguished coordinate), thereby turning it
into a metric space compatible with the parabolic homogeneity described above.

\begin{proof}[Proof of Theorem~\ref{thm:parabolic-metric}]
By definition, $D((t,x),(s,y))=0$ if and only if $t=s$ and $x=y$. To prove the
triangle inequality we need to show that
\[
\rho(t-\tau,x-z)\leq\rho(t-s,x-y)+\rho(s-\tau,y-z)
\]
for all $(t,x),(s,y),(\tau,z)\in\mathbb{R}^{n+1}$. It clearly suffices to prove that
\[
\rho(t+s,x+y)\leq\rho(t,x)+\rho(s,y)
\]
for all $(t,x),(s,y)\in\mathbb{R}^{n+1}$.

To this end, write $\lambda_{1}=\rho(t,x)$ and $\lambda_{2}=\rho(s,y)$. We claim that
\begin{equation}\label{eq:triangle-key}
\bigl|T_{1/(\lambda_{1}+\lambda_{2})}(t+s,x+y)\bigr|\leq 1
\end{equation}
implies $\lambda_{1}+\lambda_{2}\geq\rho(t+s,x+y)$, which is exactly the desired inequality. To prove the claim, recall that, by definition,
$F(t+s,x+y,\rho(t+s,x+y))=1$. Using this and the above observation about $F$, we can
rewrite \eqref{eq:triangle-key} as
\[
F\bigl(t+s,x+y,\lambda_{1}+\lambda_{2}\bigr)\leq 1=
F\bigl(t+s,x+y,\rho(t+s,x+y)\bigr).
\]
Since $F(t,x,\rho)$ is decreasing as a function of $\rho>0$, this yields
$\lambda_{1}+\lambda_{2}\geq\rho(t+s,x+y)$, as desired.

It remains to verify \eqref{eq:triangle-key}. We have
\begin{align*}
\bigl|T_{(\lambda_{1}+\lambda_{2})^{-1}}(t+s,x+y)\bigr|
&\leq
\bigl|T_{\lambda_{1}(\lambda_{1}+\lambda_{2})^{-1}}\bigl(T_{\lambda_1^{-1}}(t,x)\bigr)\bigr|
+\bigl|T_{\lambda_{2}(\lambda_{1}+\lambda_{2})^{-1}}\bigl(T_{\lambda_2^{-1}}(t,x)\bigr)\bigr| \\
&=
\frac{\lambda_{1}}{\lambda_{1}+\lambda_{2}}
\,\bigl|T_{\lambda_{1}^{-1}}(t,x)\bigr|
+\frac{\lambda_{2}}{\lambda_{1}+\lambda_{2}}
\,\bigl|T_{\lambda_{2}^{-1}}(s,y)\bigr|.
\end{align*}
By the choice of $\lambda_{1}$ and $\lambda_{2}$ and the definition of $\rho$ we have
$|T_{\lambda_{1}^{-1}}(t,x)|^{2}=F(t,x,\lambda_{1})=1$ and
$|T_{\lambda_{2}^{-1}}(s,y)|^{2}=F(s,y,\lambda_{2})=1$, hence
$|T_{\lambda_{1}^{-1}}(t,x)|=|T_{\lambda_{2}^{-1}}(s,y)|=1$. Therefore
\[
\bigl|T_{(\lambda_{1}+\lambda_{2})^{-1}}(t+s,x+y)\bigr|
\leq
\frac{\lambda_{1}}{\lambda_{1}+\lambda_{2}}
+\frac{\lambda_{2}}{\lambda_{1}+\lambda_{2}}
=1,
\]
which proves \eqref{eq:triangle-key} and completes the proof.
\end{proof}

From now on, motivated by Theorem~\ref{thm:parabolic-metric}, we shall use the
notation
\[
\|(t,x)\|=\rho(t,x),
\]
even though this quantity does not define a norm. It will be convenient, for notational
purposes, to refer to $\|(t,x)\|$ as a (parabolic) ``norm'' and to write it in this way, so as
to highlight the analogy with the situation described in Section~1.

Before ending this paragraph, let us note that a direct computation shows that there
exist constants $C_{0},C_{1}>0$ such that for all $(t,x)\in\mathbb{R}^{n+1}$ we have
\begin{equation}\label{eq:parabolic-comparable}
C_{0}\bigl(|t|^{1/2}+|x|\bigr)\leq\|(t,x)\|\leq
C_{1}\bigl(|t|^{1/2}+|x|\bigr).
\end{equation}
Here $|t|$ denotes the absolute value of $t\in\mathbb{R}$, and $|x|$ the Euclidean norm of
$x\in\mathbb{R}^{n}$. In the next subsection we shall use $|y|$ to denote the Euclidean
norm of $y\in\mathbb{R}^{m}$ for $m\in\mathbb N$.

\subsection{Lip$(1,1/2)$–type functions and their version of the implicit function theorem}

\begin{definition}
Let $\dot U\subset \dot{\mathbb R}^{n+1}$ be open. A function $f:\dot U\to{\mathbb R}^{m}$ is said to be of type $\mathrm{Lip}(1,1/2)$ in $\dot U$ if there exists a constant $M>0$ such that
\[
|f(t,x)-f(s,y)| \le M\bigl(|t-s|^{1/2}+|x-y|\bigr)
\]
for all $(t,x),(s,y)\in \dot U$.
\end{definition}

With the notation introduced above and the comparability estimate \eqref{eq:parabolic-comparable}, we have for such functions
\[
|f(t,x)-f(s,y)| \le M C_{0}^{-1}\,\|(t-s,x-y)\|,
\]
that is, $\mathrm{Lip}(1,1/2)$ functions are precisely those functions which are Lipschitz with respect to the parabolic ``norm'' $\|\cdot\|$.

\begin{theorem}\label{thm:bijection}
Let $\dot U\subset\dot{\mathbb R}^{m+1}$ be open and let $f:\dot U\to\dot{\mathbb R}^{m+1}$ be a mapping with $f(\dot U)\subset\dot{\mathbb R}^{m+1}$. Suppose there exist constants $M,K>0$ such that
\begin{itemize}
    \item[$(1)$] $\|f(\alpha)-f(\beta)\| \le M\|\alpha-\beta\|$ for all $\alpha,\beta\in \dot U$.

    \item[$(2)$] $\|f(\alpha)-f(\beta)\|\ge K\|\alpha-\beta\|$ for all $\alpha,\beta\in \dot U$.
\end{itemize}
Then $f$ is a bijection from $\dot U$ onto $f(\dot U)$, the set $f(\dot U)$ is open, and the inverse map $f^{-1}:f(\dot U)\to \dot U$ satisfies estimates of the same type as $(1)$-$(2)$.
\end{theorem}

\begin{proof}
Clearly $f:\dot U\to f(\dot U)$ is surjective. To prove injectivity, let $\alpha,\beta\in \dot U$ with $f(\alpha)=f(\beta)$. Then
\[
0=\|f(\alpha)-f(\beta)\|\ge K\|\alpha-\beta\|,
\]
and hence $\alpha=\beta$.

Thus $f$ is a bijection from $\dot U$ onto $f(\dot U)$ and has an inverse map $f^{-1}:f(\dot U)\to \dot U$. For $\alpha',\beta'\in f(\dot U)$ there exist $\alpha,\beta\in \dot U$ such that $\alpha'=f(\alpha)$ and $\beta'=f(\beta)$. Then
\[
\|f^{-1}(\alpha')-f^{-1}(\beta')\|
=\|\alpha-\beta\|
\le K^{-1}\|f(\alpha)-f(\beta)\|
=K^{-1}\|\alpha'-\beta'\|,
\]
and similarly,
\[
\|f^{-1}(\alpha')-f^{-1}(\beta')\|
=\|\alpha-\beta\|
\ge M^{-1}\|f(\alpha)-f(\beta)\|
=M^{-1}\|\alpha'-\beta'\|.
\]
Hence $f^{-1}$ satisfies inequalities of the same type as $(1)$–$(2)$. Finally, since $f^{-1}$ is continuous on $f(\dot U)$, the inverse image $(f^{-1})^{-1}(\dot U)=f(\dot U)$ is open.
\end{proof}

We now state and prove the analogue of the implicit function theorem in this setting.
\begin{theorem}[Implicit function theorem]\label{thm:IFT-Lip}
Let $\dot{U}_{m+1}\subset\dot{\mathbb{R}}^{m+1}$ and $U_{n}\subset\mathbb{R}^{n}$ be open sets. For $a\in \dot{U}_{m+1}$ and $b\in U_{n}$ let $f:\dot{U}_{m+1}\times U_{n}\to\mathbb{R}^{n}$ be a function satisfying a $\mathrm{Lip}(1,1/2)$ condition of the form
\[
|f(t,x,y)-f(s,z,w)|
\le M\bigl(|x-z|+|t-s|^{1/2}+|y-w|\bigr),
\]
and such that $f(a,b)=\vec0$. Assume in addition that
\[
|f(t,x,y)-f(t,x,w)|\ge K\,|y-w|
\]
for some $K>0$ and all $(t,x)\in \dot{U}_{m+1}$, $y,w\in U_{n}$. Then there exists an open set $\dot{V}_{m+1}\subset\dot{\mathbb{R}}^{m+1}$ with $a\in \dot{V}_{m+1}\subset \dot{U}_{m+1}$ 
and a function $\varphi:\dot{V}_{m+1}\to U_{n}$ of type $\mathrm{Lip}(1,1/2)$, i.e.,
\begin{equation}\label{eq:phi-Lip}
|\varphi(t,x)-\varphi(s,z)|
\le M'\bigl(|x-z|+|t-s|^{1/2}\bigr),
\end{equation}
such that
\[
\bigl\{(t,x,y)\in \dot{V}_{m+1}\times U_{n}: f(t,x,y)=0\bigr\}
=
\bigl\{(t,x,\varphi(t,x)):(t,x)\in \dot{V}_{m+1}\bigr\}.
\]
\end{theorem}

\begin{proof}
Choose $\varepsilon>0$ to be specified later, and consider the mapping
\[
g:\dot{U}_{m+1}\times U_{n}\to \dot{\mathbb{R}}^{m+1}\times\mathbb{R}^{n}\cong\dot{\mathbb{R}}^{m+1+n},
\qquad
g(t,x,y)=(t,x,\varepsilon f(t,x,y)).
\]
For $(t,x,y),(s,z,w)\in \dot{U}_{m+1}\times U_{n}$ we have, by \eqref{eq:parabolic-comparable},
\begin{align*}
\|g(t,x,y)-g(s,z,w)\|
&=\|(t-s,x-z,\varepsilon(f(t,x,y)-f(s,z,w)))\| \\
&\le C_{1}\bigl(|t-s|^{1/2}+|x-z|+\varepsilon|f(t,x,y)-f(s,z,w)|\bigr) \\
&\le C_{1}\bigl(|t-s|^{1/2}+|x-z|+\varepsilon M(|x-z|+|t-s|^{1/2}+|y-w|)\bigr) \\
&\le C_{1}(1+\varepsilon M)\bigl(|t-s|^{1/2}+|x-z|+|y-w|\bigr) \\
&\le p C_0^{-1}\,\|(t,x,y)-(s,z,w)\|,
\end{align*}
with $p=C_{1}(1+\varepsilon M)$. Thus $g$ is Lipschitz on $\dot{U}_{m+1}\times U_{n}$.

On the other hand,
\begin{align*}
\|g(t,x,y)-g(s,z,w)\|
&\ge C_{0}\bigl(|t-s|^{1/2}+|x-z|+|\varepsilon(f(t,x,y)-f(s,z,w))|\bigr) \\
&\ge C_{0}\bigl(|t-s|^{1/2}+|x-z|\bigr)
+ C_{0}\varepsilon\bigl|f(t,x,y)-f(s,z,w)\bigr|.
\end{align*}
We write
\[
f(t,x,y)-f(s,z,w)
=
\bigl(f(t,x,y)-f(t,x,w)\bigr)
+\bigl(f(t,x,w)-f(s,z,w)\bigr),
\]
and use the triangle inequality to obtain
\begin{align*}
\|g(t,x,y)-g(s,z,w)\|
&\ge C_{0}\bigl(|t-s|^{1/2}+|x-z|\bigr)
+ C_{0}\varepsilon|f(t,x,y)-f(t,x,w)| \\
&\quad - C_{0}\varepsilon|f(t,x,w)-f(s,z,w)| \\
&\ge C_{0}\bigl(|t-s|^{1/2}+|x-z|\bigr)
+ C_{0}K\varepsilon|y-w|
- C_{0}\varepsilon M\bigl(|x-z|+|t-s|^{1/2}\bigr) \\
&= C_{0}(1-M\varepsilon)\bigl(|t-s|^{1/2}+|x-z|\bigr)
+ C_{0}K\varepsilon|y-w|.
\end{align*}
If we choose $0<\varepsilon<M^{-1}$ and set
\[
q=\min\{C_{0}(1-M\varepsilon),\,C_{0}K\varepsilon\}>0,
\]
then
\[
\|g(t,x,y)-g(s,z,w)\|
\ge q\bigl(|t-s|^{1/2}+|x-z|+|y-w|\bigr)
\ge C_{1}^{-1}q\,\|(t,x,y)-(s,z,w)\|.
\]
Thus $g$ satisfies assumptions of Theorem \ref{thm:bijection} on $\dot{U}_{m+1}\times U_{n}$. By Theorem~\ref{thm:bijection}, $g(\dot{U}_{m+1}\times U_{n})$ is open and $g:\dot{U}_{m+1}\times U_{n}\to g(\dot{U}_{m+1}\times U_{n})$ has a bi-Lipschitz inverse with respect to $\|\cdot\|$.

Consider now the canonical projections
\[
\pi_{m+1}:\dot{\mathbb{R}}^{m+1}\times\mathbb{R}^{n}\to\dot{\mathbb{R}}^{m+1},\qquad
\pi_{m+1}(t,x,y)=(t,x),
\]
\[
\pi_{n}:\dot{\mathbb{R}}^{m+1}\times\mathbb{R}^{n}\to\mathbb{R}^{n},\qquad
\pi_{n}(t,x,y)=y.
\]
Let $\dot{V}_{m+1}$ be a Euclidean ball (with respect to the usual Euclidean metric) such that
\[
a\in \dot{V}_{m+1}\subset \dot{U}_{m+1}
\quad\text{and}\quad
\pi_{n}(g^{-1}(t,x,0))\in U_{n}
\quad\text{for all }(t,x)\in \dot{V}_{m+1}.
\]
This is possible because $g^{-1}$ and $\pi_{n}$ are continuous and $\pi_{n}(g^{-1}(a,0))=b$. Define $\varphi:\dot{V}_{m+1}\to U_{n}$ by
\[
\varphi(t,x)=\pi_{n}(g^{-1}(t,x,0)).
\]
By construction $\varphi(a)=b$. Now, by definition of $g$ for $(t,x)\in \dot{V}_{m+1}$ we have
\begin{align*}
(t,x,0)
&= g(g^{-1}(t,x,0))
= g(\pi_{m+1}(g^{-1}(t,x,0)),\,\pi_n(g^{-1}(t,x,0))) \\
&= \bigl(\pi_{m+1}(g^{-1}(t,x,0)),\,\varepsilon f(\pi_{m+1}(g^{-1}(t,x,0)),\varphi(t,x))\bigr).
\end{align*}
Since $g$ is injective we conclude
\[
(t,x)=\pi_{m+1}(g^{-1}(t,x,0))
\quad\text{and}\quad
f(\pi_{m+1}(g^{-1}(t,x,0)),\varphi(t,x))=0,
\]
hence $f(t,x,\varphi(t,x))=0$ for $(t,x)\in V_{m+1}$, that is,
\[
\{(t,x,y)\in V_{m+1}\times U_n: f(t,x,y)=0\}\supset \{(t,x,\varphi(t,x)) : x\in V_{m+1}\}.
\]
Now, if $(t,x,y)\in \dot{V}_{m+1}\times U_n$ is such that $f(t,x,y)=0$, then
$g(t,x,y)=(t,x,0)$, which implies $(t,x,y)=g^{-1}(t,x,0)$; after projecting we obtain
\[
y=\pi_n(g^{-1}(t,x,0))=\varphi(t,x),
\]
which allows us to write
\[
\{(t,x,y)\in \dot{V}_{m+1}\times U_n : f(t,x,y)=0\}
\subset
\{(t,x,\varphi(x)) : x\in \dot{V}_{m+1}\}.
\]
To finish the proof, let us see that $\varphi$ satisfies the $\mathrm{Lip}(1,1/2)$ condition.
For $(t,x)\in\dot V_{m+1}$ we have
\[
g(t,x,\varphi(t,x))=(t,x,\varepsilon f(t,x,\varphi(t,x)))=(t,x,0).
\]
Hence, if $(t,x_1),(s,x_2)\in\dot V_{m+1}$, then using \eqref{eq:parabolic-comparable} and the lower bound for $g$ obtained above, we get
\begin{align*}
C_0\,|\varphi(t,x_1)-\varphi(s,x_2)|
&\le \|(t-s,\,x_1-x_2,\,\varphi(t,x_1)-\varphi(s,x_2))\| \\
&\le \frac{C_1}{q}\,\|g(t,x_1,\varphi(t,x_1))-g(s,x_2,\varphi(s,x_2))\| \\
&= \frac{C_1}{q}\,\|(t,x_1,0)-(s,x_2,0)\| \\
&\le \frac{C_1^2}{q}\,\bigl(|t-s|^{1/2}+|x_1-x_2|\bigr).
\end{align*}
Therefore
\[
|\varphi(t,x_1)-\varphi(s,x_2)|
\le \frac{C_1^2}{C_0\,q}\,\bigl(|t-s|^{1/2}+|x_1-x_2|\bigr),
\]
which completes the proof.

\end{proof}

\subsection{Non–cylindrical domains in $\mathbb{R}^{n+1}$}

We first introduce a variant of the notion of a domain locally given by Lip$(1,1/2)$ graphs. The following definition is close to that in \cite{BrownHuLieberman1997}, and also to the definitions used, for instance, in \cite{Brown1989,LewisMurray1992,LewisMurray1995,HofmannLewis1996,HofmannLewis2005,HofmannLewis1999,Nystrom1997,Nystrom2008,ArgiolasGrimaldi2010,RiveraNoriega2014,ChoDongKim2015,DindosPetermichlPipher}.

\begin{definition}
We say that an open connected set $\Omega\subset\mathbb{R}^{n+1}$ with $\partial\Omega=\partial\overline{\Omega}$ is a Lip$(1,1/2)$ domain if for each $X_{0}\in\partial\Omega$ there exists a new coordinate system obtained from the original one by a rotation in the spatial variables $x$, together with a cylinder $C$ of the form $B\times I$ (where $B$ is a ball in $\mathbb{R}^{n}$ and $I$ is an interval in $\mathbb{R}$) that contains $X_{0}$, and a function $\psi:B\to\mathbb{R}$ of type $\mathrm{Lip}(1,1/2)$ whose local graph
\[
\Sigma(\psi,B):=\{(t,x',\psi(t,x')):(t,x')\in B\}
\]
contains $X_{0}$, such that
\[
C\cap\partial\Omega = C\cap\Sigma(\psi,B).
\]
\end{definition}

In analogy with the definition of star-like Lipschitz domains in spherical coordinates, we now define non–cylindrical star-like domains using cylindrical coordinates. For $\mathbb{R}^{n+1}$ we use cylindrical coordinates of the form $(s,r\omega)$, where $s\in\mathbb{R}$ (the time variable), $r\ge 0$ and $\omega\in S^{n-1}$. 

\begin{definition}
An open set $\Omega\subset\mathbb{R}^{n+1}$ is said to be a non–cylindrical star-like domain of class Lip$(1,1/2)$ if $\Omega=\bigl\{(s,r\omega):\ \omega\in S^{n-1},\ 0\le r<\varphi(s,\omega),\ s\in\mathbb{R}\bigr\}$, 
where $S^{n-1}$ denotes the unit sphere in $\mathbb{R}^{n}$ and $\varphi:\mathbb{R}\times S^{n-1}\to (0,\infty)$ satisfies a $\mathrm{Lip}(1,1/2)$ condition of the form
\[
|\varphi(t_{1},\omega_{1})-\varphi(t_{2},\omega_{2})|
\le M\bigl(|t_{1}-t_{2}|^{1/2}+|\omega_{1}-\omega_{2}|\bigr)
\]
and, in addition, there exist constants $\delta_{0},K_{0}>0$ such that
\[
\delta_{0}<\varphi(s,\omega)<K_{0}
\quad\text{for all }(s,\omega)\in\mathbb{R}\times S^{n-1}.
\]
The latter condition ensures that $\Omega$ remains uniformly bounded, nondegenerate and star-like in the time variable $s$, properties that are desirable, for instance, when one wants to use caloric (parabolic) measure techniques (see e.g. \cite{Nystrom1997,RiveraNoriega2003}).
\end{definition}

Note also that if we fix the time variable $s$ and define
\[
\Omega(s)=\bigl\{(s,r\omega):\ \omega\in S^{n-1},\ 0\le r<\varphi(s,\omega)\bigr\},
\]
then, for each $s$, the set $\Omega(s)$ is a star-like Lipschitz domain in the sense of Section~1.

Since $\Omega
=\bigl\{(s,r\omega):\ \omega\in S^{n-1},\ 0\le r\le\varphi(s,\omega),\ s\in\mathbb{R}\bigr\}$, 
and $\varphi$ is continuous, we have
\[
\partial\overline{\Omega}=\partial\Omega
=\bigl\{(s,\omega\varphi(s,\omega)):\ \omega\in S^{n-1},\ s\in\mathbb{R}\bigr\}.
\]

We shall use the following elementary lemma.

\begin{lemma}\label{lem:spherical}
If $\omega_{1},\omega_{2}\in S^{n-1}$ and $r_{1},r_{2}\in(0,\infty)$, then
\[
|\omega_{1}-\omega_{2}|^{2}
= r_{1}^{-1}r_{2}^{-1}\bigl(|r_{1}\omega_{1}-r_{2}\omega_{2}|^{2}-|r_{1}-r_{2}|^{2}\bigr).
\]
\end{lemma}
\begin{proof}
Using the Euclidean inner product in $\mathbb{R}^{n}$ we have
\[
|r_{1}\omega_{1}-r_{2}\omega_{2}|^{2}
= r_{1}^{2}+r_{2}^{2}-2r_{1}r_{2}\,\omega_{1}\cdot\omega_{2},
\]
\[
|r_{1}-r_{2}|^{2}
= r_{1}^{2}+r_{2}^{2}-2r_{1}r_{2},
\]
\[
|\omega_{1}-\omega_{2}|^{2}
=2-2\omega_{1}\cdot\omega_{2}.
\]
Subtracting the second equality from the first gives
\[
|r_{1}\omega_{1}-r_{2}\omega_{2}|^{2}-|r_{1}-r_{2}|^{2}
=2r_{1}r_{2}\bigl(1-\omega_{1}\cdot\omega_{2}\bigr)
=r_{1}r_{2}\,|\omega_{1}-\omega_{2}|^{2},
\]
which is the desired identity.
\end{proof}

\begin{theorem}\label{thm:noncylindrical-Lip}
If $\Omega$ is a non–cylindrical star-like domain of class $\mathrm{Lip}(1,1/2)$, then $\Omega$ is a Lip$(1,1/2)$ domain.
\end{theorem}

\begin{proof}
Since $\Omega$ is a non–cylindrical star-like domain of class $\mathrm{Lip}(1,1/2)$, it has the form
\[
\Omega=\bigl\{(s,r\omega):\ \omega\in S^{n-1},\ 0\le r<\varphi(s,\omega),\ s\in\mathbb{R}\bigr\},
\]
as specified above. In what follows we shall also write $x=(t,x',x_{n})\in\dot{\mathbb{R}}^{n+1}$, where $x'\in\mathbb{R}^{n-1}$ and $x_{n},t\in\mathbb{R}$.

Let $\alpha\in\partial\Omega$. By applying a rotation around the $t$–axis and a translation parallel to the $t$–axis we may assume, without loss of generality, that
\[
\alpha\in\partial\Omega\cap\{(0,0',\lambda):\ \lambda\in(0,\infty)\},
\]
that is, $\alpha=(0,0',\lambda)$ with $\lambda=\varphi(e_{n+1})>0$.

Fix $\eta>0$ with $\eta\le \lambda/2$, subject to additional conditions to be specified later. Let $B$ be the ball in the $(t,x')$–variables, with center at $(0,0')$ and radius $\eta$, and let $I$ be the interval in the $x_{n}$–variable, centered at $\lambda$ and of radius $\eta$. Set $U=B\times I$ and define $f:U\to\mathbb{R}$ by
\[
f(t,x',x_{n})
=\varphi^{2}\!\left(t,\frac{(x',x_{n})}{|(x',x_{n})|}\right)
-\bigl|(x',x_{n})\bigr|^{2}.
\]
Note that
\[
f(0,0',\lambda)
=\varphi^{2}(e_{n+1})-\lambda^{2}=0.
\]
To check that $f$ is of type $\mathrm{Lip}(1,1/2)$, let
\[
x=(t,x',x_{n}),\quad
y=(s,y',y_{n})\in U,
\]
and apply Lemma~\ref{lem:spherical} with
\[
\omega_{1}=\frac{x}{|x|},\quad \omega_{2}=\frac{y}{|y|},\quad
r_{1}=|x|,\quad r_{2}=|y|.
\]
We obtain
\[
\left|\frac{x}{|x|}-\frac{y}{|y|}\right|
=\Bigl(|x|^{-1}|y|^{-1}\bigl(|x-y|^{2}-\bigl||x|-|y|\bigr|^{2}\bigr)\Bigr)^{1/2}
\le \frac{2}{\lambda}\,|x-y|,
\]
since, by construction, $|x|\ge |x_{n}|\ge \lambda/2$ and similarly for $y$ if $\eta\le \lambda/2$.
Let $M>0$ be a Lip$(1,1/2)$ constant for $\varphi$. Then
\begin{align*}
|f(t,x)-f(s,y)|
&=\Bigl|\varphi^{2}\!\left(t,\frac{x}{|x|}\right)-\varphi^{2}\!\left(s,\frac{y}{|y|}\right)
+\bigl(|y|^{2}-|x|^{2}\bigr)\Bigr| \\
&\le \Bigl|\varphi^{2}\!\left(t,\frac{x}{|x|}\right)-\varphi^{2}\!\left(s,\frac{y}{|y|}\right)\Bigr|
+\bigl||x|^{2}-|y|^{2}\bigr|.
\end{align*}
Using the Lipschitz regularity of $\varphi$ and the above bound we get
\[
\Bigl|\varphi^{2}\!\left(t,\frac{x}{|x|}\right)-\varphi^{2}\!\left(s,\frac{y}{|y|}\right)\Bigr|
\le 2\|\varphi\|_{\infty}M\bigl(|t-s|^{1/2}+|x/|x|-y/|y||\bigr)
\le 2\|\varphi\|_{\infty}M|t-s|^{1/2}+\frac{4\|\varphi\|_{\infty}M}{\lambda}|x-y|,
\]
and
\[
\bigl||x|^{2}-|y|^{2}\bigr|
=\bigl||x|-|y|\bigr|\,(|x|+|y|)
\le \lambda |x-y|,
\]
again because $|x|,|y|\le K_{0}$. Altogether,
\[
|f(t,x)-f(s,y)|
\le Q\bigl(|t-s|^{1/2}+|x-y|\bigr),
\]
for some constant $Q$ depending only on $\varphi$ and $\lambda$, so $f$ is of type $\mathrm{Lip}(1,1/2)$ on $U$.

Next we verify the nondegeneracy condition in Theorem~\ref{thm:IFT-Lip}.  But also, for $y_1,y_2\in I$,
\begin{align*}
|f(t,x',y_1)-f(t,x',y_2)|
&=
\left|
\varphi^{2}\!\left(t,\frac{x'}{|(x',y_1)|},\frac{y_1}{|(x',y_1)|}\right)
-\varphi^{2}\!\left(t,\frac{x'}{|(x',y_2)|},\frac{y_2}{|(x',y_2)|}\right)
-y_1^{2}+y_2^{2}
\right| \\
&\ge |y_1+y_2|\,|y_1-y_2|
-
\left|
\varphi^{2}\!\left(t,\frac{x'}{|(x',y_1)|},\frac{y_1}{|(x',y_1)|}\right)
-\varphi^{2}\!\left(t,\frac{x'}{|(x',y_2)|},\frac{y_2}{|(x',y_2)|}\right)
\right| \\
&\ge \lambda|y_1-y_2|
-2\|\varphi\|_{\infty}M
\left|
\frac{(x',y_1)}{|(x',y_1)|}
-
\frac{(x',y_2)}{|(x',y_2)|}
\right|.
\end{align*}
Now let $(x',y_1),(x',y_2)\in U$. Considering the function $g:\mathbb{R}\to S^{n-1}$
given by
\[
g(\ell)=\frac{(x',\ell)}{|(x',\ell)|},
\]
and using the mean value theorem we estimate:
\[
\left|
\frac{(x',y_1)}{|(x',y_1)|}
-
\frac{(x',y_2)}{|(x',y_2)|}
\right|
\le
\max\left\{
\sup\left\{
\left|
\frac{\partial}{\partial x_j}
\left(
\frac{(x',\ell)}{|(x',\ell)|}
\right)
\right|,\ \ell\in[y_1,y_2]
\right\},\ 1\le j\le n
\right\}
|y_1-y_2|.
\]
But, denoting by $\delta_{j,n}$ the Kronecker delta, a direct computation gives
\[
\frac{\partial}{\partial x_j}
\left(
\frac{(x',x_n)}{|(x',x_n)|}
\right)
=
\frac{\delta_{j,n}}{|(x',x_n)|}
-
\frac{x_jx_n}{|(x',x_n)|^{3}}.
\]
If $j=n$,
\[
\frac{\delta_{j,n}}{|(x',x_n)|}
-
\frac{x_jx_n}{|(x',x_n)|^{3}}
=
\frac{1}{|(x',x_n)|}
-
\frac{x_n^{2}}{|(x',x_n)|^{3}}
\to
\frac{1}{\lambda}-\frac{1}{\lambda}=0
\quad\text{as }(x',x_n)\to(0,\lambda).
\]
If $j\ne n$,
\[
\frac{\delta_{j,n}}{|(x',x_n)|}
-
\frac{x_jx_n}{|(x',x_n)|^{3}}
=
-\frac{x_jx_n}{|(x',x_n)|^{3}}
\to 0
\quad\text{as }(x',x_n)\to(0,\lambda).
\]
In any case we can then obtain, for the previously chosen $\eta>0$,
\[
\left|
\frac{(x',y_1)}{|(x',y_1)|}
-
\frac{(x',y_2)}{|(x',y_2)|}
\right|
\le \eta|y_1-y_2|.
\]
Thus, going back to the previous estimate we obtain:
\[
|f(t,x',y_1)-f(t,x',y_2)|
\ge
\lambda|y_1-y_2|-2\|\varphi\|_{\infty}M\eta|y_1-y_2|.
\]
With all of the above, and taking $\eta$ small enough so that $\lambda-2\|\varphi\|_{\infty}M\eta>0$,
we will have that $f$ satisfies the hypotheses of Theorem~\ref{thm:IFT-Lip}.
Therefore there exists an open set $V\subset\mathbb{R}^{n}$ such that $(0,0')\in V\subset B$
and a function $\psi:V\to I$ of type $\mathrm{Lip}(1,1/2)$ such that
\[
\{(t,x',y)\in V\times I:\ f(t,x',y)=0\}=\{(t,x',\psi(t,x')):\ (t,x')\in V\}.
\]
Define the cylinder $C=V\times I$. Take $\alpha\in C\cap\Sigma(\psi,V)$, so $\alpha$ has the form
\[
\alpha=(s,x',\psi(s,x')),\qquad (s,x')\in V.
\]
By definition of $f$ we have $f(\alpha)=0$, hence
\[
\varphi^{2}\!\left(s,\frac{(x',\psi(s,x'))}{|(x',\psi(s,x'))|}\right)=|(x',\psi(s,x'))|^{2},
\]
and therefore, setting
\[
q:=|(x',\psi(s,x'))|,\qquad
\sigma:=\frac{(x',\psi(s,x'))}{|(x',\psi(s,x'))|}\in S^{n-1},
\]
we get $q=\varphi(s,\sigma)$ and
\[
\alpha=(s,x',\psi(s,x'))=(s,q\sigma)\in \partial\Omega.
\]
Thus $C\cap\Sigma(\psi,V)\subset C\cap\partial\Omega$.
Conversely, let $\beta\in C\cap\partial\Omega$. Then there exist $s\in I$ and $\omega\in S^{n-1}$ such that
\[
\beta=(s,\omega\,\varphi(s,\omega)).
\]
Evaluating $f$ at $\beta$ yields
\[
f(\beta)=\varphi^{2}(s,\omega)-|\omega\,\varphi(s,\omega)|^{2}
=\varphi^{2}(s,\omega)-\varphi^{2}(s,\omega)=0.
\]
Write $\beta=(s,x',y)$ (so $(s,x')\in V$ and $y\in I$). Since $f(s,x',y)=0$ and the implicit function theorem
gives the representation of the zero set as a graph over $V$, we must have $y=\psi(s,x')$, hence
$\beta\in C\cap\Sigma(\psi,V)$. Therefore $C\cap\partial\Omega\subset C\cap\Sigma(\psi,V)$.
Consequently, $C\cap\partial\Omega=C\cap\Sigma(\psi,V)$, 
and the proof is complete.
\end{proof}

\section*{Acknowledgements}

The first author began this work as an undergraduate student in the Mathematics program at the University of Sonora, during a research stay as part of the XX Summer Program for Scientific and Technological Research of the Pacific (2015), hosted by the second author at the Autonomous University of the State of Morelos. The authors are grateful to the University of Sonora for its support, and they also thank Martha Guzm\'an Partida for facilitating this collaboration.

\end{document}